\documentclass[12pt]{amsart}
\usepackage{amsmath,amssymb,amsthm}
\usepackage{dsfont}
\usepackage{mathrsfs}
\usepackage[margin=3.5cm]{geometry}
\usepackage{graphicx}\usepackage{subfigure}
\usepackage{multicol}
\usepackage{tikz}
\usepackage[all]{xy}
\usepackage{xfrac}
\usepackage{xcolor}
\definecolor{olive}{rgb}{0.3, 0.4, .1}
\definecolor{fore}{RGB}{249,242,215}
\definecolor{back}{RGB}{51,51,51}
\definecolor{title}{RGB}{255,0,90}
\definecolor{dgreen}{rgb}{0.,0.6,0.}
\definecolor{gold}{rgb}{1.,0.84,0.}
\definecolor{JungleGreen}{cmyk}{0.99,0,0.52,0}
\definecolor{BlueGreen}{cmyk}{0.85,0,0.33,0}
\definecolor{RawSienna}{cmyk}{0,0.72,1,0.45}
\definecolor{Magenta}{cmyk}{0,1,0,0}
\usepackage{url}
\usepackage{abstract}
\usepackage[T1]{fontenc}
\usepackage{pdfcomment} 

\usepackage{longtable}
\usepackage{fourier} 
\usepackage{array}
\usepackage{makecell}

\hypersetup{ 
    colorlinks=true,
    citecolor=Magenta,
   linkcolor=blue,
    linktoc=all
                         }

\newtheorem{defn}{Definition}[section]
\newtheorem{thm}[defn]{Theorem}
\newtheorem{lem}[defn]{Lemma}
\newtheorem{prop}[defn]{Proposition}
\newtheorem{cor}[defn]{Corollary}
\newtheorem{rmk}[defn]{Remark}
\newtheorem{qsn}[defn]{Question}

\newcommand{\R}{\mathds{R}}

\newcommand{\x}{\times}
\newcommand{\f}{\frac}
\newcommand{\D}{\mathcal{D}}
\newcommand{\K}{\mathds{k}}
\newcommand{\F}{\mathscr{F}}
\newcommand{\GG}{\mathscr{G}}
\newcommand{\PB}{\mathscr{P}_B}
\newcommand{\QB}{\mathscr{Q}_B}
\newcommand{\PV}{\mathscr{P}_V}
\newcommand{\PW}{\mathscr{P}_W}
\newcommand{\bu}{\bullet}
\newcommand{\s}{\mathscr{S}}
\newcommand{\EB}{\mathcal{E}_B}
\newcommand{\EV}{\mathcal{E}_V}
\newcommand{\EW}{\mathcal{E}_W}

\makeatletter
\@namedef{subjclassname@2020}{\textup{2020} Mathematics Subject Classification}
\makeatother

\title[Energy relative to Microlocal Projector]{Quantum Speed Limit and Categorical Energy relative to Microlocal Projector}

\author{Sheng-Fu Chiu}
\email{chiu@ncts.tw}
\address[Sheng-Fu Chiu]{National Taiwan University, Taiwan}
       

\begin{document} 
\maketitle

\begin{abstract}
Inspired by recent developments of quantum speed limit we introduce a categorical energy of sheaves in the derived category over a manifold relative to a microlocal projector. We utilize the tool of algebraic microlocal analysis to show that with regard to the microsupports of sheaves, our categorical energy gives a lower bound of the Hofer displacement energy. We also prove that on the other hand our categorical energy obeys a relative energy-capacity type inequality. As a by-product this provides a sheaf-theoretic proof of the positivity of the Hofer displacement energy for disjointing the zero section $L$ from an open subset $O$ in $T^*L$, given that $L \cap O \neq \emptyset$.
\end{abstract}

\setcounter{tocdepth}{2}
\tableofcontents

\section{Introduction}

\subsection{Quantum speed limit versus displacement energy}
Among numerous formulations of Heisenberg's famous uncertainty principle, the uncertainty between energy and time deserves a special place. This is due to the general fact that the time is rather an intrinsic variable parametrizing the system evolution than a physical quantity waiting for determination. It was the groundbreaking work of Mandelstam and Tamm \cite{MT91} that firstly claimed that the mysterious energy-time relation should be understood by an intrinsic bound, in terms of the deviation of the Hamiltonian, on how fast the evolution of any quantum system with given energy could dislocate a quantum state in a distinguishable (orthogonal) way. Margolus and Levitin \cite{ML98} pointed out that such bound could be reformulated using the averge of the Hamiltonian and coined the terminology \emph{quantum speed limit}. Later the two bounds were unified by Levitin and Toffoli \cite{LT09}. We also refer the readers to \cite{AA90} where Anandan and Aharonov developed a geometric approach towards all quantum evolutions and gave a new time-energy uncertainty principle in terms of the distance along the evolution curve measured by the Fubini-Study metric.

In their pioneering work \cite{CP18}, Charles and Polterovich revealed a wonderful mathematical link between quantum speed limit and symplectic displacement. Given a symplectic manifold $M$ and given $A,B\subset M$, we say $A$ is \emph{displaceable} from $B$ if there exists a compactly supported Hamiltonian diffeomorphism $h:I\x M\rightarrow M$ such that its time-one map satisfies $h_1(A)\cap B=\emptyset$. 
\begin{defn}
	The Hofer displacement energy for displacing $A$ from $B$ is defined to be
	\begin{equation}
	e(A,B):= \inf_{H}\{ \|H\| \,|\, h_1(A)\cap B=\emptyset \},
	\end{equation}
	where $h$ is the compactly supported Hamiltonian diffeomorphism generated by the Hamiltonian function $H: I\x M\rightarrow \R$, and the mean oscillation norm $\|H\|$ is defined to be
	\begin{equation}
	\|H\|:=\int_0^1 (\max_M H_s - \min_M H_s) ds.
	\end{equation}
\end{defn}

Charles-Polterovich's \cite{CP18} focused on the Berezin-Toeplitz quantization of a closed quantizable K\"{a}hler manifold $M$. Such quantization assigns to any Hamiltonian function $H$ on $M$ an infinite family of Toeplitz operators $T_m(H):\mathcal{H}_m\rightarrow\mathcal{H}_m$ and to each quantum state in the family $\mathcal{H}_m$  a Borel  probability measure on $M$ as its classical state. Here the index $m\in\mathbb{N}$ denotes the inverse of the Planck constant $\hbar$. They proved that (1) if the Schr\"{o}dinger flow of $T_m(H)$ asymptotically (in the limit $m\rightarrow\infty$) dislocates the quantum state then the Hamiltonian flow of $H$ displaces the support of the classical state from itself, and (2) if the Hamiltonian flow of $H$ displaces the microsupport of a quantum state then the Schr\"{o}dinger flow of $T_m(H)$ asymptotically dislocates this quantum state. In the former case, they deduced that the quantum speed limit asymptotically bounds the self-displacement energy. Furthermore, the rate of dislocation (in terms of the asymptotic behavior of $m\rightarrow\infty$) governs the competition between the rigidity and the flexibility on the quantum side of Katok's Basic Lemma (see also \cite{Katok73} and \cite{OW05}).

\subsection{Quantum speed limit beyond quantum mechanics} Our work is partially inspired by the recent theoretical breakthroughs made by Okuyama-Ohzeki's \cite{OO18} and Shanahan-Chenu-Margolus-del Campo's \cite{SCMC18} in the field of physical science. The two teams surprisingly discovered that the essence of quantum speed limit, despite its emergence from the quantum-mechanical study of the energy-time uncertainty relation, is rather a universal dynamical phenomenon for information-preserving time  evolution processes. Among them Okuyama and Ohzeki \cite{OO18} analyzed the classical Liouville equation of phase-spacial distributions to obtain a \emph{fundamental speed limit}, and Shanahan, Chenu, Margolus, and del Campo \cite{SCMC18} investigated the motion of the Wigner phase representations and determined the bounds of semiclassical and classical speeds in terms of the norm of the Moyal product. 

A remarkable feature of their results is that such fundamental speed limits hold nontrivially even in the limit of vanishing Planck constant. Similar to their precedents in the field, both of approaches in \cite{OO18,SCMC18} utilized the formalism of unitary evolutions on Hilbert spaces. 

The present paper is an attempt to answer the following natural question:
\begin{qsn}
Other than the Hilbert space formalism, is there any theory regarding states and their evolutions which manages to treat the fundamental speed limit in a way more consistent with the topology of symplectic displacement? Moreover, is such theory capable to deal with symplectic rigidity problems beyond self-displacements?
\end{qsn}

We would like to work in the realm of certain differential graded categories $\D(Q\x\R)$ of sheaves over a manifold $Q$ called Tamarkin category (see Subsection \ref{defTamarkin} and Remark \ref{hbar}). In this formulation, a Hilbert space $\mathcal{H}$ is replaced with the Tamarkin category $\D(Q\x\R)$, and quantum states as vectors in $\mathcal{H}$ are replaced with sheaves as objects in $\D(Q\x\R)$. Likewisely, the scalar field of $\mathcal{H}$ is replaced with the category of chain complexes $D(\text{pt})$. For two states $\F$ and $\GG$ in $\D(Q\x\R)$ there is no inner product but rather a chain complex $Rhom(\F,\GG)$ of homomorphisms, and instead of genuine orthogonality there are left and right derived orthogonalities.

For a time-parametrized family of states $\{\F_s\}_{s\in I}$, the fidelity between the initial $\F_0$ and the final $\F_1$ may be defined as the complex $Rhom(\F_0,\F_1)$. The notion of the fundamental speed limit is about estimating the minimal time for the fidelity to vanish, in terms of the given Hamiltonian, which is universal for any initial state. Dually, for symplectic topological purpose we ask for the least energy for the fidelity $Rhom$ to vanish, in terms of the given initial state, which is universal for any dislocating evolution in a unit of time. Moreover, to treat symplectic problems beyond self-displacements, we require the evolution to dislocate our initial state from another given subsystem. In this paper such a subsytem is a geometrically related subcategory of $\D(Q\x\R)$ and can be characterized by a \emph{microlocal projector}.

Given an open subset $B$ its microlocal projector $\PB$ is by definition a functor that projects $\D(Q\x\R)$ to a subcategory $\D_B(Q\x\R)$. Here $\D_B(Q\x\R)$ is the left semiorthogonal complement of a natural subcategory $\D_{T^*Q\setminus B}(Q\x\R)$ satisfying certain microlocal conditions (see Definition \ref{Tamarkinsub} and Definiton \ref{admissible}). We ask for the energy needed to evolve from one sheaf to another sheaf that is semiorthogonal (vanishing fidelity) to $\D_B(Q\x\R)$. This leads to the notion of \emph{relative categorical energy}, in the form of a function $e_B: \D(Q\x\R)\rightarrow \R_{\geq0}$ (see Definition \ref{categorical energy}).

\begin{center}

	\begin{longtable}{ |c | c | c | c|}
		
		\hline
		\textbf{\makecell{Categorification \\of dynamics}} & \textbf{\makecell{Quantum\\ dynamics}} & \textbf{\makecell{Classical\\ dynamics}} & \textbf{\makecell{Sheaf theory}} \\ \hline
		
		\textbf{State}	& $\psi\in\mathcal{H}$ & Lagrangian $A\subset T^*Q$& $\F\in\D(Q\x\R)$  \\ \hline
		
		\textbf{Evolution}	& \makecell{Unitary operator\\ $U:\mathcal{H}\rightarrow\mathcal{H}$ }& \makecell{Hamiltonian\\ diffeomorphism \\ $h:T^*Q\rightarrow T^*Q$ } & \makecell{Convolution functor \\$\bu\s:\D(Q\x\R)\rightarrow\D(Q\x\R)$ } \\ \hline
		
		\textbf{Subsystem} &	Subspace  $\mathcal{V}\subset\mathcal{H}$  & Open ball $B\subset T^*Q$ & \makecell{Subcategory \\ $\D_B(Q\x\R)\subset\D(Q\x\R)$}  \\ \hline
		
		\textbf{Fidelity}	& $L^2$-inner product & \makecell{Set-theoretic\\ intersection} & \makecell{Derived\\ homomorphisms}  \\  \hline 
		
		\makecell{\textbf{Displacement}	\\ \textbf{from given}\\ \textbf{ subsystem}}& $U(\psi)\in \mathcal{V}^\perp$ & $h(A)\cap B =\emptyset$ & \makecell{$\F\bu\s\in {\D_{B}}^{\perp,l}$ } \\ \hline
		
		\textbf{ \makecell{Energy of \\ given \\evolution}}	 & \makecell{ Operator norm\\ $\int_0^1 \|U_s\|_{\text{op}}ds$} & \makecell{Hofer norm\\$\int_0^1 \| f_s\|_{\text{Hofer}} ds$}  &  \makecell{ Interleaving distance\\ $(-)\mapsto d_{\text{inter}}((-)\bu\s,-)$ } \\ \hline
		
	\end{longtable}

\end{center}

Our first nolvelty is that it is not $\PB$ but its right adjoint functor $\EB$ that enables us to define the categorical energy $e_B$. Note that the evolutions we care about in symplectic geometry are driven by classical Hamiltonians. Namely, given a Hamiltonian flow $h_s$ on $T^*Q$ there exists a sheaf of kernel $\s_s$ parametrized by the time $s$, such that the convolution operation with $\s_s$ gives a categorical action on $\D(Q\x\R)$. This recipe is called sheaf quantization, mainly due to the fact that the diffeomorphism $h_s$ moves the micro-supports of sheaves under the categorical evolution respectively (see Theorem \ref{quantization} and Remark \ref{quantizationrmk}). In this situation our categorical evolutions are by no means unitary in the usual sense as their Hilbert space counterparts are. So it is crucial to point out that for our purpose this right adjoint functor $\EB$ coorperates with the sheaf quantization better than $\PB$ does.

In addition to the fidelity $Rhom$ and the evolution $\s$, another piece of ingredience we need is energy measurement. This is achieved by estimating the persistency acquired by a state before and after an evolution acting on the category $\D(Q\x\R)$. The sheaf theoretic measurement of persistency relies on the notion of interleaving distance firstly introduced by Kashiwara and Schapira \cite{KS18} and further adapted by Asano and Ike \cite{AI20} to a non-balanced version. In Asano-Ike's \cite{AI20} they extended Tamarkin's non-displaceability theorem in a quantitative way and remarkably related a persistence-like categorical distance to the Hofer displacement energy. In our definition of $e_B$ this persistence-like distance is used to measure how much energy is needed to eliminate the fidelity of the initial sheaf where the fidelity relative to $B$ is detected by applying the functor $\EB$. This way the fundamental speed limit is put under the (dualized) context of sheaf-theoretic persistency.

The second novelty of our approach is the independency of Tamarkin's separation theorem, an important result on which Asano-Ike's theory \cite{AI20} is premised. Moreover, in the present paper we make no assumption to any positivity result from hard symplectic geometry while such assumption was made in the proof of \cite{CP18} by Charles and Polterovich that dislocations yield displacements. Another difference between our approach and \cite{AI20,CP18} is that we are able to estimate the Hofer displacement energy of mixed types such as displacing a Lagrangian from an open set. With the aforementioned categorical energy $e_B(-)$ relative to $B$, our main result is:

\begin{thm}[Theorem \ref{comparison}]
Let $B$ be an open bounded subset of $T^*Q$. Suppose $A$ is a closed subset of $T^*Q$, then for any object $\F\in\D_A(Q\x\R)$ the following categorical-Hofer inequality holds:
\begin{equation}
e_B(\F)\leq e(A,B).
\end{equation}
\end{thm}

This leads to a direct formulation of a nontrivial lower bound of the Hofer displacement energy of mixed type using algebraic microlocal analysis:
 
\begin{thm}[Relative energy-capacity inequality, Theorem \ref{energy-capacity}]
Let $L$ be a smooth manifold of dimension $n$. Suppose $\jmath:B(r)\hookrightarrow T^*L$ is a symplectically embedded open ball of $T^*L$ relative to the zero section $L$ (that is, $ \jmath^{-1}(L)=\R^n \cap  B(r)$).  Let $U=\jmath(B(r))$, then in $T^*L$ one has
\begin{equation}
e(L,U)\geq \f{1}{2} \pi r^2.
\end{equation}

In particular, suppose $O$ is an open subset of $T^*L$ such that $O\cap L\neq\emptyset$. Then $e(L,O)$ is strictly positive.\end{thm}

\subsection{The roles of $\R$ and $\hbar$}\label{hbar}

The extra $\R$-direction in $\D(Q\x\R)$ was proposed by Tamarkin \cite{Tamarkin18} in order to upgrade sheaf theoretic method to treat non-conic Lagrangians. Moreover, the translation along $\R$ induces a natural structure of Novikov ring action and thus strengthens the link between sheaf theory and Fukaya category. 

In the microlocal sheaf intepretation of deformation quantization theory, this extra $\R$-variable may be regarded as the real version of the Fourier-Laplace dual of the Planck constant $\hbar$. When it comes to $\mathscr{D}$-module theory (by taking  $\hbar\rightarrow 1$), the corresponding microlocal sheaf theory specializes to sheaves on $Q$ and shelters the informationless conic dimension. We refer the reader to Kuwagaki's works \cite{Ku20,Ku22} for a thorough explanation of the  relations between $\hbar$ and $\R$. We shall  mention that in \cite{Ku20} Kuwagaki introduced a $\R$-equivariant version of Tamarkin's theory and applied it to treat non-exact Lagrangians such as spectral curves of $\hbar$-deformation quantizations over the Novikov ring. In \cite{Ku22} he proposed a $\hbar$-enhanced Riemann-Hilbert correspondence using equivariant sheaves.

In the present paper we focus on the original Tamarkin category without passing to $\R$-equivariant sheaves. Our work can be seen as an attempt to categorify the dynamics across quantum and classical mechanics on the most elementary level. We look forward to the future research which will involve the $\R$-equivariant framework and the $\hbar$-deformation quantization theory.

\section{Review of microlocal theory of sheaves}

\subsection{Sheaves and their microlocal supports}

Let $Q$ be a smooth manifold. We denote by $(q,p)$ the points of the cotangent bundle $T^*Q$ where $q\in Q$ and $p\in T^*_q Q$. Let $\K$ be a ground ring and let $D(Q)$ be the derived differential graded category of sheaves of $\K$-modules on $Q$. We do not require objects of $D(Q)$ to be cohomologically bounded.

Given any object $\F$ of $D(Q)$, Kashiwara and Schapira \cite{KS90} defined its \textit{microsupport} (an abbreviation of microlocal singular support, denoted by $SS(\F))$ by the closure of those $(q_0,p_0)\in T^*Q$ such that there exists a $C^1$ function $\phi:Q\rightarrow\mathbb{R}$ satisfying $d\phi(q_0)=p_0$ and the following ill propagation condition:
\begin{equation}
(R\Gamma_{\{q\in Q|\phi(q)\geq\phi(q_0)\}}\F)_{q_0}\ncong0 .
\end{equation}

In other words, $SS(\F)\subset T^*Q$ is the minimal closed $\R_{>0}$-conic subset containing all singular codirections of $\F$. Here by \emph{singular} we mean the complement of all the codirections along which the derived sections of $\F$ extend microlocally. It is easy to see that the application $SS$ is additive with respect to the direct sums ($SS(\F\oplus\GG)=SS(\F) \bigcup SS(\GG)$) and is independent of the cohomological shift of gradings ($SS(\F[1])=SS(\F)$). Incidentally, we would like to recall the readers that the notion of microsupport is genuinely derived: one in general cannot fully recover the subset $SS(\F)$ simply from the collection of the subsets $\{SS(H^j(\F)) \,|\, j\in\mathbb{Z}  \}$.

\subsection{The Tamarkin categories $\D$}\label{defTamarkin}

We are peculiarly interested in the sheaf-theoretic approximation of geometric objects in $T^*Q$ by microsupports. It is among the main results of \cite{KS90} that for any $\F$ in $D(Q)$ the corresponding microsupport $SS(\F)$ is always a closed $\R_{>0}$-conic coisotropic subset of $T^*Q$. On the other hand, many coisotropic subsets of interest in symplectic geometry and topology do not respect the $\R_{>0}$-conic structure, nor do the Hamiltonian diffeomorphisms that act on them. In \cite{Tamarkin18}, Tamarkin tackled this problem by considering certain subcategories of the category $D(Q\x\R)$ in which the behaviors of the subsets of $T^*Q$ are more tractable than in the category $D(Q)$. 

We denote the coordinate of $T^*(Q\x\R)$ by $(q,p,z,\zeta)$ where $(q,p)\in T^*Q$, $z\in\R$, and $\zeta\in T^*_z \R$. Following \cite{Tamarkin18}, we consider the full subcategory
\begin{equation}
 D_{\zeta\leq0}=\{\mathcal{F}\in D(Q\x\R) | SS(\F)\subset \{\zeta\leq0\}\}.
\end{equation}
of $D(Q\x\R)$ characterized by being microlocally singular at the negative extra codirection.

\begin{defn}[Tamarkin Category]
We define 
$\D(Q\x\R):= D_{\zeta\leq0}(Q\times\R)^{\perp,l}$. Here the left semiorthogonal complement is taken with respect to $Rhom$ in the background category $D(Q\times\R)$.
\end{defn}

An astonishing fact about the Tamarkin category $\D(Q\x\R)$ is that there exists a quotient functor $D(Q\x\R)\rightarrow \D(Q\x\R)$ which can be represented by a sheaf kernel convolution. Before describing this kernel, let us introduce the convolution product $\bu$ of sheaf kernels as follows.

Let $\F\in D(Q_1\x Q_2 \x\R)$ and $\GG \in D(Q_2\x Q_3 \x\R)$. We define 
\begin{equation}
\F\bu_{Q_2}\GG= R\pi_{13\sigma !}({\pi_{Q12R1}}^{-1}\F\otimes{\pi_{Q23R2}}^{-1}\GG)\in D(Q_1\x Q_3\x\R)
\end{equation}
where $\pi_{Q12R1}:Q_1\x Q_2\x Q_3\x\R_1\x\R_2\rightarrow Q_1\x Q_2\x\R_1$ and $\pi_{Q23R2}:Q_1\x Q_2\x Q_3\x\R_1\x\R_2\rightarrow Q_2\x Q_3\x\R_2$ are projections and $\pi_{Q13R\sigma}:Q_1\x Q_2\x Q_3\x\R_1\x\R_2\rightarrow Q_1\x Q_3\x \R$ maps $(q_1,q_2,q_3,z_1,z_2)$ to $(q_1,q_3,z_1+z_2)$.

When $Q_2$ and $Q_3$ are single points, by abuse of notation we denote by $\F\bu\GG$ the convolution product $\F\bu_{Q_2}\GG$ for $\F\in D(Q\x\R)$ and $\GG\in D(\R)$. This way convoluting with a given object of $D(\R)$ gives rise to a functor $D(Q\x\R)\rightarrow D(Q\x\R)$.

For any nonempty interval $A\subset \R$, denote by $\K_A\in D(\R)$ the constant sheaf supported by $A$. We also denote $\K_0=\K_{\{z=0\}}$, $\K_{\geq c}=\K_{[c,\infty)}$ and $\K_{>c}=\K_{(c,\infty)}$. For any $\F\in D(Q\x\R)$ we have $\F\bu\K_0=\F$. Moreover:

\begin{prop}\cite[Proposition 2.1, Proposition 2.2]{Tamarkin18}\label{Tamarkin}
Convolution with the exact triangle 
\begin{equation}
\K_{\geq0}\rightarrow\K_0\rightarrow\K_{>0}[1]\xrightarrow{+1}
\end{equation}
gives rise to the semiorthogonal decomposition with respect to the triple ($\D(Q\x\R)$, $D(Q\x\R)$, $D_{\zeta\leq0}(Q\times\R)$). To be more precise, for any $\F\in D(Q\x\R)$, we have the following exact triangle
\begin{equation}
	\F\bu\K_{\geq0}\rightarrow\F\bu\K_0(=\F)\rightarrow\F\bu\K_{>0}[1]\xrightarrow{+1}
\end{equation}
such that $\F\bu\K_{\geq0}\in \D(Q\x\R)$ and $\F\bu\K_{>0}[1]\in D_{\zeta\leq0}(Q\times\R)$.

\end{prop}

Next we introduce the notion of Tamarkin subset categories for constrained conditions on microsupports as follows:
\begin{defn}[Tamarkin Subset Category]\label{Tamarkinsub}
	Consider the de-homogenization map $\rho:T^*Q\x\dot{T}^*\mathbb{R}\rightarrow T^*Q$, where 
	$\dot{T}^*\mathbb{R}=\{(z,\zeta)\in T^*\mathbb{R} | \zeta\neq0  \}$ 
	and $\rho$ sends $(q,p,z,\zeta)$ to $(q,\frac{p}{\zeta})$. 
	\begin{enumerate}
		\item For $A \overset{\text{cls}}\subset T^*Q$, we define $\D_A(Q\x\R):=\{\F\in \D| SS(\F)\subset \rho^{-1}(A)\}$.
		\item	For $U\overset{\text{open}}\subset T^*Q$, we define $\D_U(Q\x\R):={\D_{T^*Q\setminus U}}^{\perp,l}$ w.r.t $Rhom$ in $\D$.
	\end{enumerate}
\end{defn}

Note that the Tamarkin category for open subset is defined in terms of the left semiorthogonal complement of its counterpart. This is due to the general fact that the microsupport of a sheaf is a closed subset.

\subsection{Sheaf quantization of Hamiltonian diffeomorphisms}

Let $I$ be an open interval of $\R$ containing the origin and let $h: I\x T^*Q\rightarrow T^*Q$ be a Hamiltonian diffeomorphism driven by the Hamiltonian function $H:T^*Q \rightarrow \R$. Recall the de-homogenization map $\rho:T^*Q\x\dot{T}^*\mathbb{R}\rightarrow T^*Q$ given by $\rho(q,p,z,\zeta)=(q,\frac{p}{\zeta})$. We could consider the homogenized Hamiltonian function defined on $T^*Q\x\dot{T}^*\R$:
\begin{equation}
\begin{split}
\tilde{H_s} &:=  \zeta H_s\circ\rho, \\
\tilde{H_s}(q,p,z,\zeta) &:=   \zeta  H_s(q,p/\zeta)
\end{split}
\end{equation}

In general this homogenized function $\tilde{H}$ could be undefined at $\zeta=0$. For the sheaf quantization technique to work one needs $\tilde{H}$ to generate a Hamiltonian flow that is furthermore defined for $p\neq0$ even when $\zeta=0$. This is treated in \cite{GKS12,GS14} by imposing the compact support condition:

\begin{prop}\cite[Proposition A.6(ii)]{GKS12}\cite[Proposition 3.2]{GS14}\label{GKS}
	Assume that the Hamiltonian diffeomorphism $h$ is compactly supported, then $h$ can be extended to a homogeneous Hamiltonian diffeomorphism $\tilde{h}: I\x \dot{T}^*(Q\x\R)\rightarrow \dot{T}^*(Q\x\R)$ such that the following diagram commutes:
\begin{center}
	\begin{tabular}{c}
		\xymatrix{
			{{I\x T_{\zeta>0}^*}(Q\x\R)}\ar[d]_{\rho\x id_I} \ar[rr]^{\tilde{h}}  &   & {{T_{\zeta>0}^*}(Q\x\R)}\ar[d]_{\rho}\\
			{I\x T^*Q }\ar[rr]^{h}&  & {T^*Q}.
		}
	\end{tabular}
\end{center}
	More precisely, there are smooth functions $u$ on $I\x T^*Q$ and $v$ on $I\x \pi_0(\dot{T}^*Q)$ such that for $\zeta>0$, one has
\begin{equation}
\tilde{h}_s((q,p;z,\zeta))=(h_s(q,p);z+u_s(q,\f{p}{\zeta}),\zeta)
\end{equation}
and for $\zeta=0$ one has
\begin{equation}
\tilde{h}_s((q,p;z,0))=(q,p;z+v_s([(q,p)]),0).
\end{equation}	
Note that when $(q,\f{p}{\zeta})\notin \text{supp}(h)$ one can take $u_s(q,\f{p}{\zeta})= v_s([q,\f{p}{\zeta}])$.
\end{prop}

Let $M=Q\x\R$ and denote $T^*M=\{(x,\xi)\}$. Let $\tilde{h}:I\x \dot{T}^*M\rightarrow \dot{T}^*M$ be the homogeneous Hamiltonian diffeomorphism obtained in Proposition \ref{GKS}. Let $\tilde{H}$ be the homogeneous Hamiltonian function generating the Hamiltonian diffeomorphism $\tilde{h}$. Consider a conic Lagrangian submanifold $\Lambda\subset T^*I\x\dot{T}^*M\x\dot{T}^*M$ defined by
\begin{equation}\label{Lagsuspension}
\Lambda=\{((s,-\tilde{H}_s(\tilde{h}_s(x,\xi));\tilde{h}_s(x,\xi);(x,-\xi))| s\in I    ,(x,\xi)\in \dot{T}^*M )    \}.
\end{equation}

The submanifold $\Lambda$ is called the \textit{Lagrangian suspension} of the Hamiltonian diffeomorphism $\tilde{h}$ . More precisely, let $\Lambda':=\{(s;\tilde{h}_s(x,\xi);(x,\xi))\}\subset I\x \dot{T}^*M\x\dot{T}^*M$ be the total graph of $\tilde{h}$ and let $\pi_I:T^*I\rightarrow I$ be the projection, we have:

\begin{prop}\cite[Lemma A.2]{GKS12}
\begin{enumerate}
	\item $\Lambda$ is the unique conic Lagrangian such that $(\pi_I\x\emph{id}_{\dot{T}^*M\x\dot{T}^*M})|_\Lambda\cong \Lambda' $.
	\item For arbitrary $s\in I$, the inclusion $i_s:M\x M\hookrightarrow I\x M\x M$ is non-characteristic and the graph of $\tilde{h}_s$ at $s$ is equal to the symplectic reduction $\Lambda\circ T^*_s I$ of $\Lambda$ at $ T^*_s I$.
\end{enumerate}
\end{prop}

A remarkable achievement of \cite{GKS12} is the following theorem due to Guillermou, Kashiwara and Schapira:

\begin{thm}\cite[Theorem 3.7]{GKS12}\label{quantization}		
	Associated with the Hamiltonian diffeomorphism $\tilde{h}$ there exists a (unique) sheaf $\mathcal{K}^{\tilde{h}} \in D(I\x M\x M)$ satisfying
\begin{enumerate}
	\item  $SS(\mathcal{K}^{\tilde{h}})= \Lambda$ outside the zero section.
	\item  $\mathcal{K}^{{\tilde{h}}}|_{s=0}\cong \K_{\Delta_{M}}:=\K_{\Delta_{M}}:=\K_{\{(q,q,z,z)\}}$.
\end{enumerate}
\end{thm}

In this paper we prefer to work with the Tamarkin category $\D(Q\x\R)$ and we prefer a convolution kernel in $\D(I\x Q\x Q\x\R)$ rather than a kernel in $D(I\x M\x M)=D(I\x Q\x\R\x Q\x\R)$. Again let $h: I\x T^*Q\rightarrow T^*Q$ be a (non-homogeneous) Hamiltonian diffeomorphism. Let 
$$ 
I\x Q_1\x Q_2\x\R \xleftarrow{\pi_{IQ1Q2R1}}   I\x Q_1\x\R_1\x Q_2 \x\R_2  \xrightarrow{ \pi_{R2}  }  \R
$$ be projections. We set 
\begin{equation}
\s:=  R{\pi_{IQ1Q2R1}}_!( \mathcal{K}^{\tilde{h}}\otimes {\pi_{R2}}^{-1}(\K_{\geq0}))\in \D(I\x Q\x Q\x\R).
\end{equation}

\begin{rmk}\label{quantizationrmk}
The kernel $\s$ is called the sheaf quantization of $h$ for the following reasons:
\begin{enumerate}
	\item Each $\s_s:=\s|_{\{s\}\x Q\x Q\x\R}$ defines an autoequivalence $(-)\bu\s_s$ on  $\D(Q\x\R)$.
	\item Moreover $(-)\bu\s_s$ induces a functor $\D_A(Q\x\R)\rightarrow \D_{\tilde{h}_s(A)}(Q\x\R)$.
	\item Family of paths of Hamiltonian diffeomorphisms with the common endpoint $\tilde{h}_1$ induce the same $\s|_{s=1}$. 
	
\end{enumerate}
\end{rmk}

For an explanation for Remark \ref{quantizationrmk}(3), we refer the reader to \cite[Proposition 4.18]{Zhang20}.

\subsection{Interleaving distance on $\D(Q\x\R)$}

The behaviour of the convolution product relies heavily on the additive structure of the extra variable $\R$. Let $T_c : z\mapsto z+c$ be the translation by $c$ units along $\R$. By abuse of notation we also denote by $T_c$ its action on $D(Q\times\mathbb{R})$. A noteworthy feature of the Tamarkin category $\D(Q\x\R)$ is we have for $c\geq0$, a natural transformation
\begin{equation}
\tau_c : \text{Id}\Rightarrow T_c
\end{equation}
between the endofunctors of $\D(Q\x\R)$. Indeed, let $\K_{\geq c}\in D(\R)$ be the constant sheaf supported by the subset $\{z\geq c\}$, then the closed inclusion $\{z\geq c\}\hookrightarrow\{z\geq0\}$ induces a morphism $\K_{\geq 0}\rightarrow \K_{\geq c}$ in $D(\R)$. It is easy to see that convoluting objects of $\D(Q\x\R)$ with $ \K_{\geq c}$ is equivalent to the convolution with $\K_{c}$, in which the latter gives rise to the action $T_c$ on the Tamarkin category $\D(Q\x\R)$. Therefore as kernels of convolution the map $(-)\bu\K_{\geq 0}\rightarrow (-)\bu\K_{\geq c}$ induces the above natural transformation $\tau_c$.

Let $\F,\GG$ be objects of the Tamarkin category $\D(Q\x\R)$ and let $a,b\geq0$. Following \cite{AI20}, we say the pair $(\F,\GG)$ is $(a,b)$-\emph{interleaved} if there exists morphisms $\F \overset{\alpha}{\underset{\delta}{\rightrightarrows}}T_a\GG$ and $\GG \overset{\beta}{\underset{\gamma}{\rightrightarrows}} T_b\F$ satisfying the following equations in morphisms of $\D(Q\x\R)$:
\begin{equation}\label{interleaving}
\begin{cases}
[\F\xrightarrow{\alpha}T_a\GG\xrightarrow{T_a\beta} T_{a+b}\F ]=\tau_{a+b}(\F),\\
[\GG\xrightarrow{\gamma}T_b\F\xrightarrow{T_b\delta} T_{a+b}\GG ]=\tau_{a+b}(\GG).
\end{cases}
\end{equation}

Note that $\tau_c(\F)\in \hom(\F,T_c\F)=H^0 Rhom(\F,T_c\F)$. In general it could happen that a derived morphism in $Rhom$ is not cohomologous to zero while holding zero cohomology in $\hom$. Therefore being interleaving is indeed a nontrivial condition which measures the mutual persistency between two sheaves.

The brilliant idea of distance for sheaves was raised by Kashiwara and Schapira in \cite{KS18}, based on the convolution product in the derived setting. For its motivation and historical remarks we refer the reader to \cite{KS18} and the references therein. Inspired by \cite{KS18}, Asano and Ike \cite{AI20} established a persistence-like distance function on the Tamarkin category with the notion of (non-balanced) sheaf-theoretic interleaving (\ref{interleaving}):
\begin{defn}
	For $\F,\GG\in \D(Q\x\R)$ their interleaving distance $d(\F,\GG)$ is defined as follows
	\begin{equation}
	d(\F,\GG):= \inf\{a+b \,|\,(\F,\GG) \text{\, is \,} (a,b)\text{-interleaved} \}.
	\end{equation}
\end{defn}

The following important theorem, due to Asano and Ike, demonstrates the stability property of the interleaving distance under Hamiltonian diffeomorphisms with respect to the mean oscillation norm. 

\begin{thm}\cite[Theorem 4.16]{AI20}\label{bound}
Let $H:I\x T^*Q \rightarrow\R $ be a compactly supported Hamiltonian function and denote by $h$ the Hamiltonian diffeomorphism generated by $H$. Let $\s\in \D(I\x Q\x Q\x\R)$ be the sheaf quantization associated with $h$. Let $\F\in\D(Q\x\R)$ and set $\F_s:=\F\bu\s_s \in\D(Q\x\R)$ for $s\in I$. Then $d(\F_0,\F_1)\leq \|H \| = \int_{0}^{1} (\max H_s - \min H_s) ds  $.
\end{thm}

\section{Categorical energy relative to microlocal projector}

\subsection{Microlocal projector associated to open set} 

\begin{defn}[Admissible open sets and their  projectors]\label{admissible}
Let $B$ be an open subset of $T^*Q$. We say $B$ is admissible if there is an exact triangle 
\begin{equation}
\PB\rightarrow\K_{\Delta}\rightarrow \QB \xrightarrow{+1}
\end{equation}
in $\D(Q\x Q\x\R)$
such that the convolution operation with the above triangle gives rise to the semiorthogonal decomposition with respect to the triple of subset categories (see Definition \ref{Tamarkinsub}) 
$$
	(\D_{B}(Q\x\R),\D(Q\x\R),\D_{T^*Q\setminus B}(Q\x\R)).
$$
This means that the embedding $\D_{B}(Q\x\R)\hookrightarrow \D(Q\x\R)$ admits a right adjoint functor $(-)\bu\PB$, and the embedding $\D_{T^*Q\setminus B}(Q\x\R)\hookrightarrow \D(Q\x\R)$ admits a left adjoint functor $(-)\bu\QB$.

We call $\PB$ the microlocal projector associated with $B\subset T^*Q$. Sometimes the pair ($\PB$,$\QB$) is called microlocal kernel pair associated with $B$.

\end{defn}

\begin{thm}\label{projector} If the open set $B\subset T^*Q$ is bounded, then it is admissible.
\end{thm}

The proof of Theorem \ref{projector} is a straightforward \textit{verbatim} repetition of the one in \cite[Theorem 3.11]{Chiu17} and will be omitted. 

Here we make a remark on the connection of microlocal projector and sheaf quantization. In view of the smooth Urysohn's Lemma for the close subset $T^*Q\setminus B$, there exists a smooth function $0\leq F \leq 1$ on $T^*Q$ such that $T^*Q\setminus B=F^{-1}(1)$. This way $B$ coincides to the sublevel set $\{F<1\}$. The boundedness of $B$ guarantees the Hamiltonian $I$-action generated by the autonomous Hamltonian function $F$, where $I=\R$. By Theorem \ref{quantization} and Remark \ref{quantizationrmk}, there exists an associated sheaf quantization $\s\in\D(I\x Q\x Q\x\R)$. It turns out that $\PB$ can be constructed from $\s$. For the details of the construction of the projector we refer the reader to \cite{Chiu17,Zhang21,Zhang20}. 

We would also like to remark that the projector $\PB$ only depends on the open subset $B\subset T^*Q$ and does not depend on the choice of $F$. A different choice of the function $F$ leads to a  reparametrization of the flow (that is, the Hamiltonian $I$-action) and gives isomorphic $\PB$.

\subsection{Functor right adjoint to microlocal projector}

In the present paper we would like to emphasize that behind the proof of \cite[Theorem 3.11]{Chiu17} there is a functor $\EB:\D(Q\x\R)\rightarrow\D(Q\x\R)$ which is the right adjoint functor of $(-)\bu\PB$ with respect to $R\hom$. That is, for any objects $\F$ and $\GG$ of $\D(Q\x\R)$ the following adjoint relation holds:
$$
R\hom(\GG\underset{Q}{\bu}\PB,\F)=R\hom(\GG,\EB(\F)).
$$

To be more precise, with Grothendieck's six functor formalism one defines for any object $\F\in\D(Q\x\R)$, the object
\begin{equation}
\EB(\F)=R{\pi_{Q1R1}}_* \underline{Rhom}({\pi_{Q12R2}}^{-1}(\PB),{\pi_{Q2\sigma}}^! (\F)).
\end{equation}

Here the maps $\pi_{Q1R1}:Q_1\x Q_2\x\R_1\x\R_2\rightarrow Q_1\x\R_1$ and $\pi_{Q12R2}:Q_1\x Q_2\x\R_1\x\R_2\rightarrow Q_1\x Q_2\x\R_2$ are projections. The map $\pi_{Q2\sigma}:Q_1\x Q_2\x\R_1\x\R_2\rightarrow Q_2\x\R$ projects along $Q_1$ and takes the summation $\R_1\x\R_2\rightarrow\R$.

The key feature of $\EB$ is the following microlocal vanishing property outside the open sublevel set $B$, which plays an important role in the proof of Theorem \ref{projector} (as in \cite[Theorem 3.11]{Chiu17}) that the convolution with $\PB$ projects to the left semiorthogonal complement of $\D_{T^*Q\setminus B} (Q\x\R)$:

\begin{prop}\label{vanishing} 
	Let $B$ be an admissible open subset of $T^*Q$. Then for all $\F\in\D_{T^*Q\setminus B} (Q\x\R)$, one has $\EB(\F)\cong 0$.
\end{prop}

\begin{cor}
	Let $B$ be an admissible open subset of $T^*Q$. Then the full subcategory $\D_{B}$ also serves as the right semiorthogonal complement of its right semiorthogonal complement $\D_{T^*Q\setminus B}$.

\begin{center}
	\begin{tabular}{c}
				\xymatrix{\ar @{} [drr] |{\bigcup}   
			& {\D} \ar@/_/[dl]_{\bu\PB}  \ar@/^/[dr]^{\EB} &  \\
			{{\D_{T^*Q\setminus B}}^{\perp,l}}  &  {\D_{B}} \ar@2{-}[l] \ar@2{-}[r] &{ {\D_{T^*Q\setminus B}}^{\perp,r}}}
		
	\end{tabular}
\end{center}

\end{cor}

It is worth mentioning that the functor $\EB$ is naturally compatible with the family $T$ of translation functors. Take any $c\geq0$ and let $T_c$ be  the corresponding translation endofunctor on the Tamarkin category $\D(Q\x\R)$. Then the functors $\EB$ and $T_c$ commute with each other (that is, $\EB\circ T_c=T_c\circ\EB$). In particular, the following Lipschitz property holds for the functor $\EB$:

\begin{lem}\label{functorial}
	Suppose $(\F_0,\F_1)$ is $(a,b)$-interleaved in $\D(Q\x\R)$ and let $B$ be an admissible open subset of $T^*Q$. Then the pairwise image $(\EB(\F_0),\EB(\F_1))$ under $\EB$ is also $(a,b)$-interleaved in $\D(Q\x\R)$.
\end{lem}

\subsection{Energy of sheaf relative to microlocal projector}

\begin{defn}[Categorical energy]\label{categorical energy}
	Let $B$ be an admissible open subset of $T^*Q$. For any object  $\F\in\D(Q\x\R)$ we define its categorical energy, relative to $B$, by the number
	\begin{equation}
	e_B(\F) := d(0,\EB(\F)).
\end{equation}
\end{defn}

It is the categorical energy $e_B(-)$ that measures how orthogonal a sheaf is to the subcategory $\D_B(Q\x\R)$. In particular $e_B(\F)=0$ if $\F$ is orthogonal to $\D_B(Q\x\R)$ with respect to the sheaf-theoretic inner product $R\hom$.

Recall that the displacement energy of two subsets $A$ and $B$ of a symplectic manifold is defined via the variation over all Hamiltonian disjunctions:
$$
e(A,B):= \inf_{H}\{ \|H\| \,|\, h_1(A)\cap B=\emptyset \},
$$
where $h$ denotes the Hamiltonian diffeomorphism generated by $H$, and the mean oscillation norm of $H$ is given by 
$$
\|H\|:=\int_0^1 (\max H_s - \min H_s) ds .
$$

The following comparison theorem justifies the idea that our notion of the categorical energy $e_B(-)$ captures the symplectic energy necessary to "drag" a sheaf away from being microsupported somewhere in the open subset $B$.

\begin{thm}\label{comparison}
Suppose $A$ is a closed subset of $T^*Q$ and $B$ is admissible, then for any object $\F\in\D_A(Q\x\R)$ we have the following categorical-Hofer inequality
\begin{equation}
e_B(\F)\leq e(A,B).
\end{equation}
\end{thm}

\begin{proof}
	
	 Let $H:I\x T^*Q\rightarrow \R$ be a Hamiltonian function. We assume that $H$ generates a compactly supported Hamiltonian diffeomorphism $h:I\x T^*Q\rightarrow T^*Q$ such that its time-one map $h_1$ disjoints $A$ from $B$. Let $\s\in\D(I\x Q\x\R\x Q\x\R)$ be the sheaf quantization associated with $h$ which can be obtained from Theorem \ref{quantization} (see also Remark \ref{quantizationrmk}). Then for any object $\F\in\D_A(Q\x\R)$ we have $\F\circ\s_1\in\D_{h_1(A)}(Q\x\R)$.
	 
	 By the disjunction condition $h_1(A)\cap B=\emptyset$ and the microlocal vanishing property (Proposition \ref{vanishing}) one get $\EB(\F\circ\s_1)=0$. Therefore the categorical energy of $\F$, relative to $B$, is given by	 
	 \begin{equation}
	 e_B(\F)=d(0,\EB(\F))=d(\EB(\F\circ\s_1),\EB(\F)).
	 \end{equation}
	 From the compatibility lemma (Lemma \ref{functorial}) and the stability theorem (Theorem \ref{bound}) we deduce that 
	 \begin{equation}\label{inequality}
	 e_B(\F)\leq d(\F\circ\s_1,\F)\leq \|H\|.
	 \end{equation}
	 Therefore by taking the infimum of $\|H\|$ over all such Hamiltonians $H$ we get
	\begin{equation}
	e_B(\F)\leq e(A,B).
	\end{equation}
\end{proof}

Let us close this subsection with the following remark. In the proof of Theorem \ref{comparison}, the first part of the inequality (\ref{inequality}) 
$$
e_B(\F)\leq d(\F\circ\s_1,\F)
$$ 
is exhibited in a form of capacity-energy-like inequality. Indeed, in \cite[Lemma 7, Remark 15, Remark 16]{HLS15} it is shown that Lisi-Rieser's relative capacity (we refer the reader to \cite{LR20} for its original definition) does obey a  capacity-energy type inequality. Needless to say, such an inequality always plays an essential role in the symplectic $C^0$-rigidity phenomenon such as the rigidity of coisotropic submanifolds and their characteristic foliations, which is among the main results of \cite{HLS15}. It would be very intriguing to see whether our categorical notion $e_B(-)$ serves as a sheaf-theoretic version of Lisi-Rieser's capacity of the open subset $B$ relative to a given sheaf $\F$.

\section{Displacement energy of Lagrangian from open set}

This section is devoted to the sheaf-theoretic estimate of the Hofer  displacement energy of mixed type, namely a Lagrangian intersecting with an open subset. Let us start with the nontriviality of the relative categorical energy $e_B$, as shown in the following example:

\begin{prop}\label{eg}
Let $Q$ be the Euclidean space of dimension $n$. Let $B=B(r)=\{q^2+p^2<r^2\}$ be a Darboux open ball of $T^*Q$ and let $\F=\K_{Q\x\R_{\geq0}}\in\D(Q\x\R)$ be the sheaf quantization of the zero section. Then one has 
$$
e_B(\F)\geq \f{1}{2}\pi r^2.
$$
\end{prop}

\begin{proof}

We recall that there is a periodic structure of the microlocal ball projector as described in \cite[Section 3.4]{Chiu17}, where it is first introduced and applied to the proof of the nonsqueezing property of contact balls. We would also like to refer the reader to  \cite{Guillermou20} for its application to the nonsqueezing property of symplectic balls, and a related structure for the sheaf quantization associated with the constant speed geodesic flow.

Let $S:I\x Q\x Q\rightarrow \R$ be the generating function of the Hamiltonian flow associated with $F(q,p)=q^2+p^2$. For $0<s<\f{\pi}{2}$ it has the form
\begin{equation}
S_s(q_1,q_2)=\f{\cos(2s)}{2\sin(2s)}({q_1}^2+{q_2}^2) - \f{1}{\sin(2s)}q_1 q_2.
\end{equation}

Consider a function $f:(0,\f{\pi}{2})\x Q_1\x Q_2 \rightarrow \R$ given by 
\begin{equation}
f(s)(q_1,q_2)=-S_s(q_1,q_2)-sr^2
\end{equation}
and consider the domain $N\subset Q\x Q$ defined by
\begin{equation}
 N=\{(q_1,q_2)\,|\, |q_1|\leq r, |q_2|\leq r  \}.
\end{equation}

Fix a point $(q_1,q_2)\in N$. The critical point of $f(s)$ satisfies the following equation
\begin{equation}
-r^2=\f{\partial S}{\partial s}= -\f{{q_1}^2+{q_2}^2}{\sin^2(2s)}+ \f{2\cos(2s)q_1 q_2}{\sin^2(2s)},
\end{equation}
which is a quadratic equation in $\cos(2s)$:
\begin{equation}
r^2 \cos^2(2s)-2q_1q_2\cos(2s)+({q_1}^2+{q_2}^2-r^2)=0.
\end{equation}

Note that for $|\xi|>1$, the quantity
\begin{equation}
r^2 \xi^2-2q_1q_2\xi+({q_1}^2+{q_2}^2-r^2)=(\xi q_1-q_2)^2 + (r^2-{q_1}^2)(\xi^2-1)
\end{equation}
remains positive since the equalities $|q_1|=r$ and $\xi q_1 = q_2$ cannot hold simultaneously (otherwise they would lead to $|q_2|>r$). So for the critical equation to possess a solution we must have $|\xi|\leq1$, for which we can assign some critical point $0\leq s\leq\f{\pi}{2}$ of $f$ such that $\cos(2s)=\xi$. The function $f(s)$ has two critical points  $s_1$ and $s_2$. Here we fixed their order by setting $0\leq s_1\leq s_2\leq \f{\pi}{2}$. It is always true that $f(s_1)\geq f(s_2)$.

Let us define a subset $\Sigma$ of $Q\x Q\x \R$ by
\begin{equation}\label{Sigma}
\Sigma=\{(q_1,q_2,z)| (q_1,q_2)\in N, f(s_2)(q_1,q_2)\leq z < f(s_1)(q_1,q_2)\}.
\end{equation}

Note that although the function $f(s)$ is defined on $(0,\f{\pi}{2})$, we can still include certain degenerated values $f(0)$ and $f(\f{\pi}{2})$ in the definition of $\Sigma$. For example when $q_1=q_2=q$ it is easy to see that the quantity
$$
f(s)(q,q)=\f{1-\cos(2s)}{\sin(2s)}q^2 -sr^2
$$ 
approaches to $0$ as $s$ approaches to $0$. The critical equation of $f(s)(q,q)$ possesses solutions satisfying $\cos(2s)=1$ or $\cos(2s)=\f{2q^2-r^2}{r^2}$. Therefore we have $s_1=0$ and $s_2=\arccos(\f{|q|}{r})$. Similarly for $(q_1,q_2)=(q,-q)$ one has $s_1=\arcsin(\f{|q|}{r})$ and $s_2=\f{\pi}{2}$. In particular, at  $(q_1,q_2)=(0,0)$ we have $s_1=0$ and $s_2=\f{\pi}{2}$. In this case we write $f(s_1)=0$ and $f(s_2)=-\f{\pi}{2}r^2$. 

This way we have
\begin{equation}\label{persistence}
\K_{\Sigma}|_{\{(0,0)\}\x\R} = \K_{\{ -\f{\pi}{2}r^2\leq z < 0  \}}.
\end{equation}

With the above notion we define (in an essentially unique way) an object $\Gamma\in\D(Q\x Q\x\R)$ to be the nontrivial extension between two opposite constant sheaves:
\begin{equation}
\K_\Sigma \rightarrow \Gamma \rightarrow T_{\f{\pi}{2}r^2} \K_{(\Psi(\Sigma))}[-n] \xrightarrow{+1}
\end{equation}
where $\Psi:(q_1,q_2)\mapsto(q_1,-q_2)$ is the reflection with respect to the skew diagonal of $Q\x Q$. This is due to the simple fact that the Hamiltonian flow after $\f{\pi}{2}$ units of time gives a half-rotation. The object $\Gamma$ is obtained by gluing $\K_\Sigma$ and $T_{\f{\pi}{2}r^2} \K_{(\Psi(\Sigma))}$ along the diagonal $\{(q,q,0) \,|\, q\in Q \}$ of $Q\x Q\x\{0\}$. The fundamental class of the diagonal contributes to the cohomological shift of degree $[-n]$.

Let $\PB$ be the microlocal projector associated with the Darboux ball $B=B(r)$. The periodic structure of $\PB$ is exhibited as the following, for any $m\in\mathbb{Z}$,
\begin{equation}
\PB|_{Q\x Q\x [\f{(2m-1)\pi}{2}r^2,\f{(2m+1)\pi}{2}r^2)}\cong T_{m\pi r^2}(\Gamma) [-mn].
\end{equation}
With this periodic structure the projector $\PB$ is constructed in the sense of homotopy colimit by iteratively gluing on $\Gamma$ along the skew diagonal $\{(q,-q,\f{(2m-1)\pi}{2}) \,|\, q\in Q\}$ of $Q\x Q\x\{\f{(2m-1)\pi}{2}\}$, for $m$ running through $\mathbb{Z}$.

Now we turn to the computation of $\EB(\F)$ where $\F=\K_{Q\x\R_{\geq0}}$. In view of the adjoint relation 
$$
R\hom(\GG\bu\PB,\F)=R\hom(\GG,\EB(\F))
$$
there is a way to interpret the right adjoint functor $\EB$ in terms of its left adjoint, the microlocal projector $\PB$, by unwrapping the convolution over $Q\x\R$:
\begin{equation}\label{badexpression}
\EB(\F)=R{\pi_{Q1R1}}_* \underline{Rhom}({\pi_{Q12R2}}^{-1}(\PB),{\pi_{Q2\sigma}}^! (\F)).
\end{equation}

Here the maps $\pi_{Q1R1}:Q_1\x Q_2\x\R_1\x\R_2\rightarrow Q_1\x\R_1$ and $\pi_{Q12R2}:Q_1\x Q_2\x\R_1\x\R_2\rightarrow Q_1\x Q_2\x\R_2$ are projections. The map $\pi_{Q2\sigma}:Q_1\x Q_2\x\R_1\x\R_2\rightarrow Q_2\x\R$ projects along $Q_1$ and takes the summation $\R_1\x\R_2\rightarrow\R$.

\begin{center}
	\begin{tabular}{c}
				\xymatrix{
			& {Q_1\x Q_2\x\R_1\x\R_2} \ar [d] _{\pi_{Q2\sigma}} \ar@/_/[dl]_{\pi_{Q1R1}}  \ar@/^/[dr]^{\pi_{Q12R2}} &  \\
			{Q_1\x\R_1}  &  {Q_2\x\R} &{ Q_1\x Q_2\x\R_2}.
			}
		
	\end{tabular}
\end{center}

For the sake of computation it is more convenient to stay all the way in the Tamarkin categories $\D$. It turns out that the functor $\underline{Rhom}(-,-)$ is expected to invert the micro-support of the first argument, hence in the above expression (\ref{badexpression}) of $\EB(\F)$ the object $\PB$ (which is situated in $\D$) is not suitable for the functor $\underline{Rhom}$ to begin with. We hereby apply a trick introduced by Guillermou and Schapira in \cite{GS14} to deal with this issue with $ \underline{Rhom}$. Consider an auxiliary isomorphism $\theta:Q_1\x Q_2\x\R_1\x\R_2\rightarrow Q_1\x Q_2\x\R_1\x\R_2$ given by $\theta(q_1,q_2,z_1,z_2)=(q_1,q_2,z_1+z_2,-z_2)$ and let $i:\R_2\rightarrow\R_2$ be the inversion map $i(z_2)=-z_2$. The map $i$ also stands for a product map $i:Q_1\x Q_2\x\R_2\rightarrow Q_1\x Q_2\x\R_2$ with $i(q_1,q_2,z_2)=(q_1,q_2,-z_2)$. It is easy to see that $\theta$ and $i$ satisfy the relations $\theta\circ \theta=Id$, $\pi_{Q1R1}\circ \theta=\pi_{R2\sigma}$, and $i\circ \pi_{Q12R2}\circ \theta= \pi_{Q12R2}$. It is also clear that at the functor level we have $\theta^{-1}=\theta^!=R\theta_*$ and they commute with the $\underline{Rhom}$ functor. 

With the above notations and it follows from (\ref{badexpression}) that
\begin{equation}
\begin{split}
\EB(\F) &
=R{\pi_{Q1R1}}_* \underline{Rhom}({\pi_{Q12R2}}^{-1}(\PB),{\pi_{Q2\sigma}}^! (\F))\\
& =R{\pi_{Q1R1}}_* \underline{Rhom}(\theta^{-1}\circ{\pi_{Q12R2}}^{-1}\circ i^{-1}(\PB), \theta^!\circ{\pi_{Q1R1}}^!(\F))\\
& =R{\pi_{Q1R1}}_* \theta^{-1} \underline{Rhom}({\pi_{Q12R2}}^{-1}\circ i^{-1}(\PB),  {\pi_{Q1R1}}^!(\F) )\\
& =R{\pi_{Q1R1}}_*R\theta_* \underline{Rhom}({\pi_{Q12R2}}^{-1}\circ i^{-1}(\PB),  {\pi_{Q1R1}}^!(\F) )\\
& =R{\pi_{Q2\sigma}}_*  \underline{Rhom}({\pi_{Q12R2}}^{-1} i^{-1}(\PB),  {\pi_{Q1R1}}^!(\F) )  .
\end{split}
\end{equation}

Denote by 
\begin{equation}
{\mathcal{E}_{B,\Sigma}}(\F):= R{\pi_{Q2\sigma}}_*  \underline{Rhom}({\pi_{Q12R2}}^{-1}i^{-1}(\K_\Sigma),  {\pi_{Q1R1}}^!(\F) )
\end{equation}
the object obtained by replacing the term $\PB$ in the above expression  $\EB(\F)$ with $\K_\Sigma$ (see (\ref{Sigma}) for the definition of $\Sigma$).

From the assumption $\F=\K_{Q\x\R_{\geq0}}$ we obtain 
\begin{equation}
{\pi_{Q1R1}}^!(\F)=\K_{Q_1\x Q_2\x\R_{\geq0}\x\R_2}[n+1] ,
\end{equation}
which is itself a constant sheaf concentrated at a single cohomological degree. For simplicity let us write $f_j=f(s_j)(q_1,q_2)$ and $\Sigma=\{f_2\leq z<f_1\}$ and $i(\Sigma)=\{-f_1<z\leq -f_2\}$. The following computation 
\begin{equation}
\begin{split}
{\mathcal{E}_{B,\Sigma}}(\F) & = R{\pi_{Q2\sigma}}_*  \underline{Rhom}({\pi_{Q12R2}}^{-1}i^{-1}(\K_\Sigma),  {\pi_{Q1R1}}^!(\F) )\\
& = R{\pi_{Q2\sigma}}_*  \underline{Rhom}(\K_{i(\Sigma)\x\R_1} , \K_{Q_1\x Q_2\x\R_{\geq0}\x\R_2}[n+1])\\
& = R{\pi_{Q2\sigma}}_*  \K_{\{(q_1,q_2,z_1,z_2)| z_1\geq0,  -f_1(q_1,q_2)\leq z_2 < -f_2(q_1,q_2)   \}}[n+1]\\
&= \K_{\{ (q_1,q_2,z)|  -f_1(q_1,q_2)\leq z < -f_2(q_1,q_2)  \}}[n+1]
\end{split}
\end{equation}
shows that the object ${\mathcal{E}_{B,\Sigma}}(\F)$ is again a constant sheaf concentrated at a single cohomological degree. It follows that the object $\EB(\F)$ is a gradual glue from copies of ${\mathcal{E}_{B,\Sigma}}(\F)$ in the same way as that of $\PB$ from copies of $\K_\Sigma$.

Since the interleaving relation (\ref{interleaving}) concerns only  equations within the degree zero cohomology $\hom(-,T_c(-))$ of the derived morphisms, it is sufficient to show that for $0\leq c< \f{\pi}{2}r^2$ the natural transformation $\tau_c$ gives a nonzero morphism on the object ${\mathcal{E}_{B,\Sigma}}(\F) $. This is achieved by restricting $\tau_c$ to the submanifold $\{(0,0)\}\x\R$ of $Q\x Q\x\R$ as follows.

Recall from (\ref{persistence}) we have $f_1(0,0)=0$ and $f_2(0,0)=-\f{\pi}{2}r^2$ and hence
$$
{\mathcal{E}_{B,\Sigma}}(\F)  |_{\{(0,0)\}\x\R}=\K_{\{ 0\leq z < \f{\pi}{2}r^2 \}}.
$$
It is then obvious that for $0\leq c< \f{\pi}{2}r^2$ the morphism 
$$
\tau_c(\K_{\{ 0\leq z < \f{\pi}{2}r^2 \}}) \in \hom(\K_{\{ 0\leq z < \f{\pi}{2}r^2 \}},T_c\K_{\{ 0\leq z < \f{\pi}{2}r^2 \}})
$$ 
is given by a nontrivial restriction morphism induced by the closed inclusion $[c,\f{\pi}{2}r^2) \subset [0,\f{\pi}{2}r^2)$. In fact the class of this restriction persists through $0\leq c< \f{\pi}{2}r^2$.

\end{proof}

In what follows, we transport the estimate in Proposition \ref{eg} from Euclidean spaces to general cotangent bundles:

\begin{thm}[Relative energy-capacity inequality]\label{energy-capacity}
Let $L$ be a smooth manifold of dimension $n$. Suppose $\jmath:B(r)\hookrightarrow T^*L$ is a symplectically embedded open ball of $T^*L$ relative to the zero section $L$ (that is, $ \jmath^{-1}(L)=\R^n \cap  B(r)$).  Let $U=\jmath(B(r))$, then in $T^*L$  one has
\begin{equation}
e(L,U)\geq \f{1}{2} \pi r^2.
\end{equation}

In particular, suppose $O$ is an open subset of $T^*L$ such that $O\cap L\neq\emptyset$. Then $e(L,O)$ is strictly positive.
\end{thm}

\begin{proof}

Let $B=B(r)$ be the Darboux open ball of radius $r$ centered at the origin of $Q=\R^n$. Let $\jmath: B\hookrightarrow T^*L$ be a symplectic embedding and denote its image by $U=\jmath(B)$ and let $Y=U\cap L$. Let $X=B \cap Q=\{|q|<r\}$ then $\jmath$ is an isomorphism between the pairs $(B,X)$ and $(U,Y)$. We denote by $\imath:X\rightarrow Y$ the restriction of $\jmath$ on $X$. Note that while $\imath$ is an isomorphism from $X$ to $Y$, it is not necessary for $\jmath$ to sends the cotangent subbundle structure of $T^*X$ to that of $T^*Y$.

For $0<\rho<r$, let $W=B(\rho)$ be a smaller open ball of radius $\rho$ centered at the origin of $Q$ and let $V=\jmath(W)$. We have $W\Subset  B$ and $W\cap Q\Subset X$. Similarly we have $V\Subset U$ and $V\cap L\Subset Y$. We write $W=\{F<1\}$ as a sublevel set for a function $F$ and write $U=\jmath (W)=\{G<1\}$ for a function $G$, where $F=G\circ\jmath$.

Given that $\jmath$ is a sympletic isomorphism, $\jmath$ takes the Hamiltonian $\R_s$-action of $F$ to that of $G$ and thus preserves the corresponding Lagrangian suspensions of the total graphs $\Lambda$ (\ref{Lagsuspension}). In view of Theorem \ref{quantization}, this means that $\imath$ respects their sheaf quantizations, that is $\s_F=(\imath\x\imath)^{-1}\s_G$. According to the construction (see the discussion below Theorem \ref{projector}) of the microlocal projectors it is easy to see that $\PW=(\imath\x\imath)^{-1}\PV$. Hence for $\F\in\D(L\x\R)$ one has $(\imath^{-1}\F)\underset{Q}{\bu}\PW=(\imath^{-1}\F)\underset{Q}{\bu}((\imath\x\imath)^{-1}\PV)=\imath^{-1}(\F\underset{L}{\bu}\PV)$. The application of the adjunction to $\imath$ gives
\begin{equation}
\begin{split}
{Rhom}(\F,\imath_*\EW(\K_{X\x\R_{\geq0}})) &  \cong{Rhom}(\imath^{-1}\F,\EW(\K_{X\x\R_{\geq0}}))\\
& \cong{Rhom}((\imath^{-1}\F)\underset{Q}{\bu}\PW,\K_{X\x\R_{\geq0}})\\
&\cong {Rhom}( \imath^{-1}(\F\underset{L}{\bu}\PV),\K_{X\x\R_{\geq0}}  )\\
& \cong {Rhom}(\F\underset{L}{\bu}\PV, \K_{Y\x\R_{\geq0}})\\
& \cong{Rhom}(\F, \EV(\K_{Y\x\R_{\geq0}})).
\end{split}
\end{equation}
This means that
\begin{equation}\label{fun1}
\imath_*\EW(\K_{X\x\R_{\geq0}})\cong\EV(\K_{Y\x\R_{\geq0}}).
\end{equation}
Similarly we have 
\begin{equation}\label{fun2}
\EW(\K_{X\x\R_{\geq0}})\cong  \imath^{-1}\EV(\K_{Y\x\R_{\geq0}}).
\end{equation}

Fix a $c\geq0$. Substitute (\ref{fun1}) and (\ref{fun2}) into the following natural adjunction 
\begin{equation}
{Rhom}(\EV(\K_{Y\x\R_{\geq0}}), \imath_* T_c \EW(\K_{X\x\R_{\geq0}})  )\overset{\mu}{\cong} {Rhom}(\imath^{-1}\EV(\K_{Y\x\R_{\geq0}}), T_c \EW(\K_{X\x\R_{\geq0}})  )
\end{equation}
one gets an equation
\begin{equation}\label{mu1}
{Rhom}(\EV(\K_{Y\x\R_{\geq0}}),  T_c \EV(\K_{Y\x\R_{\geq0}})  ) \overset{\mu}{\cong} {Rhom}(\EW(\K_{X\x\R_{\geq0}}), T_c \EW(\K_{X\x\R_{\geq0}})  ).
\end{equation}

On the other hand, the open inclusion $X\subset Q$ induces the exact triangle 
\begin{equation}\label{help}
 \K_{X\x\R_{\geq0}} \rightarrow \K_{Q\x\R_{\geq0}} \rightarrow \K_{(Q\setminus X)\x\R_{\geq0}} \xrightarrow{+1}.
\end{equation}

Since $ \K_{(Q\setminus X)\x\R_{\geq0}} \in\D_{T^*Q\setminus B}(Q\x\R)$, by Proposition \ref{vanishing} we have 
\begin{equation}
\EW(\K_{(Q\setminus X)\x\R_{\geq0}})=0.
\end{equation}

Then apply the functor $\EW$ to the exact triangle (\ref{help}) we get
\begin{equation}\label{WQ}
\EW( \K_{X\x\R_{\geq0}} )\cong \EW( \K_{Q\x\R_{\geq0}} ),
\end{equation} 

and likewisely for the open inclusion $Y\subset L$ we get 
\begin{equation}\label{VL}
\EV( \K_{Y\x\R_{\geq0}} )\cong \EV( \K_{L\x\R_{\geq0}} ).
\end{equation} 

Substitute (\ref{WQ}) and (\ref{VL}) into the adjunction equation (\ref{mu1}) then it follows that
\begin{equation}\label{mu2}
{Rhom}(\EV(\K_{L\x\R_{\geq0}}),  T_c \EV(\K_{L\x\R_{\geq0}})  ) \overset{\mu}{\cong} {Rhom}(\EW(\K_{Q\x\R_{\geq0}}), T_c \EW(\K_{Q\x\R_{\geq0}})  ),
\end{equation}
and in the zeroth cohomology of (\ref{mu2}) there is a natural isomorphism 
\begin{equation}
\mu:\tau_c(\EV(\K_{L\x\R_{\geq0}})) \mapsto \tau_c(\EW(\K_{Q\x\R_{\geq0}})).
\end{equation}

Thus from the construction $W=B(\rho)$ and Proposition \ref{eg} one obtains
\begin{equation}\label{one}
e_V(\K_{L\x\R_{\geq0}}) = e_W(\K_{Q\x\R_{\geq0}}) \geq \dfrac{1}{2}\pi \rho^2.
\end{equation}
On the other hand, from Theorem \ref{comparison} and the fact that $V\subset U$, we deduce that
\begin{equation}\label{two}
e(L,U)\geq e(L,V)\geq  e_V(\K_{L\x\R_{\geq0}}).
\end{equation}
By (\ref{one}) and (\ref{two}) we get the lower bound estimate in terms of $\rho$:
\begin{equation}\label{three}
e(L,U) \geq \dfrac{1}{2}\pi \rho^2.
\end{equation}
Since (\ref{three}) is valid for arbitrary $0<\rho<r$, it follows that 
\begin{equation}\label{four}
e(L,U) \geq \dfrac{1}{2}\pi r^2 ,
\end{equation}
as claimed.

For $O\cap L \neq\emptyset$ one picks a point $x\in O\cap L$. For $r$ small enough there exists a symplectic open embedding $\jmath_r : B(r)\hookrightarrow O$ which is centered at $x$ and relative to $L$. By (\ref{four}) we have $e(L,O)\geq e(L,\jmath_r(B(r)))\geq \f{1}{2}\pi r^2 >0$.

\end{proof}

\section*{Acknowledgements} Part of the work on this paper was carried out while the author was visiting Tel Aviv University in the spring of 2018. He hereby thanks Tel Aviv University and the event organizers for their hospitality during the visit. He warmly thanks Leonid Polterovich and Jun Zhang for many stimulating discussions. He is also grateful to Tomohiro Asano and Yuichi Ike for their carefully reading the manuscript and pointing out some mistakes.

\bibliographystyle{amsplain}

\end{document}